\documentclass[12pt,a4paper]{amsart}
\usepackage[latin1]{inputenc}
\usepackage[english]{babel}
\usepackage{amsmath,amstext,amsthm,amsfonts,amssymb,amscd}
\usepackage[dvips]{graphicx}
\usepackage{color}
\usepackage{hyperref}

\usepackage{a4wide}

\usepackage{cleveref}

\theoremstyle{plain}

\makeatletter
\@namedef{subjclassname@2020}{\textup{2020} Mathematics Subject Classification}
\makeatother

\newtheorem{theorem}{Theorem}[section]
\newtheorem{proposition}[theorem]{Proposition}

\newtheorem{lemma}[theorem]{Lemma}
\newtheorem{corollary}[theorem]{Corollary}

\newtheorem{maintheorem}{Theorem}

\newcommand{\cmt}{\begin{maintheorem}}
\newcommand{\fmt}{\end{maintheorem}}

\newtheorem{maincorollary}[maintheorem]{Corollary}

\newcommand{\cmc}{\begin{maincorollary}}
\newcommand{\fmc}{\end{maincorollary}}

\theoremstyle{remark}

\newtheorem{remark}[theorem]{Remark}

\newcommand{\topo}{\operatorname{top}}

\begin{document}
	
	\title{Equilibrium states for hyperbolic potentials via inducing schemes}
	
	
	
\author[J. F. Alves]{Jos\'{e} F. Alves}
\address{Jos\'{e} F. Alves\\ Departamento de Matem\'{a}tica\\ Faculdade de Ci\^encias da Universidade do Porto\\ Rua do Campo Alegre 687\\ 4169-007 Porto\\ Portugal}
\email{jfalves@fc.up.pt} \urladdr{http://www.fc.up.pt/cmup/jfalves}
	
	\author[K. Oliveira]{Krerley Oliveira}
	\address{Krerley Oliveira, Instituto de Matem\'atica, Universidade Federal de Alagoas, 57072-090 Macei\'o, Brazil} \email{krerley@gmail.com}
	
	\author[E. Santana]{Eduardo Santana}
	\address{Eduardo Santana, Universidade Federal de Alagoas, 57200-000 Penedo, Brazil}
	\email{jemsmath@gmail.com}
	
	\dedicatory{Dedicated to Jacob Palis for his contribution to Dynamical Systems}
	
	\subjclass[2020]{37D35, 37D99}
	\keywords{Expanding measures, equilibrium states, nonuniform expansion, measures of maximal entropy, inducing schemes}

	\thanks{J.F. Alves   was partially supported by CMUP (UID/MAT/00144/2019) and
PTDC/MAT-PUR/4048/2021, which are funded
by FCT (Portugal) with national (MEC) and European structural funds
through the program  FEDER, under the partnership agreement PT2020.}

	\date{\today}

	
	
	\maketitle
	
	\begin{abstract} 
		In a context of non-uniformly expanding maps, possibly with the presence of a critical set, we prove the existence of finitely many ergodic equilibrium states 
		for hyperbolic potentials. Moreover, the equilibrium states are expanding measures. This generalizes a result due to Ramos and Viana, where  analytical methods are used for maps with no critical sets. 
		The strategy here consists in using a finite number of   inducing schemes with a  Markov structure
		in infinitely many symbols to code the dynamics, to obtain an equilibrium state for the associated symbolic dynamics and then projecting it to obtain an equilibrium state for the original map. We apply our results to the important class of multidimensional  Viana maps.
	\end{abstract}

	\bigskip

	\tableofcontents


	\section{Introduction}
	
	The theory of equilibrium states on dynamical systems was firstly developed by Sinai, Ruelle and Bowen in the sixties and seventies. 
	It was based on applications of techniques of Statistical Mechanics to smooth dynamics.
	Given a continuous map $f: M \to M$ on a compact metric space $M$ and a continuous potential $\phi : M \to \mathbb{R}$, an  {\textit{equilibrium state}} is an 
	invariant measure that satisfies a variational principle, that is, a measure $\mu$ such that
	$$
	\displaystyle h_{\mu}(f) + \int \phi d\mu = \sup_{\eta \in \mathcal{M}_{f}(M)} \bigg{\{} h_{\eta}(f) + \int \phi d\eta \bigg{\}},
	$$
	where $\mathcal{M}_{f}(M)$ is the set of $f$-invariant probabilities on $M$ and $h_{\eta}(f)$ is the so-called metric entropy of $\eta$.

	In the context of uniform hyperbolicity, which includes uniformly expanding maps, equilibrium states do exist and are unique if the potential is H\"older continuous
	and the map is transitive. 
	Beyond uniform hyperbolicity, the theory is still far from complete and	 was studied by several authors, including Bruin, Keller, Demers, Li, Rivera-Letelier, Iommi and Todd 
	\cite{BT,BK,DT,IT1,IT2,LRL} for interval maps; Denker and Urbanski  \cite{DU} for rational maps; Leplaideur, Oliveira and Rios 
	\cite{LOR} for partially hyperbolic horseshoes; Buzzi, Sarig and Yuri \cite{BS,Y}, for countable Markov shifts and for piecewise expanding maps in one and higher dimensions. 
	For local diffeomorphisms with some kind of non-uniform expansion, there are results due to Oliveira \cite{O}; Arbieto, Matheus and Oliveira \cite{AMO};
	Varandas and Viana \cite{VV}. All of whom proved the existence and uniqueness of equilibrium states for potentials with low oscillation. Also, for this type of maps,
	Ramos and Viana  \cite{RV}  proved it for potentials so-called  {\textit{hyperbolic}}, which includes the previous ones. The hyperbolicity of the potential is
	characterized by the fact that the pressure emanates from the hyperbolic region.
	
	The conclusion of our main result is similar to that of \cite{RV}, but while the results  in~\cite{RV} are specific to local homeomorphisms, we allow transformations with critical sets. Our strategy is completely different, 
	since we do not use the analytical approach of the transfer operator in order to obtain conformal measures. In Section~\ref{Lifted}, we use results on countable Markov shifts by Sarig for the ``coded'' dynamics in inducing schemes given by Pinheiro in~\cite{Pi1}, where a Markov structure is constructed. 
	By using such a structure, we lift ergodic probabilities  which have pressure close to the topological pressure of the original map (the hyperbolicity of the potential help us to prove that such probabilities are expanding measures) and use the thermodynamic formalism for shifts to obtain a unique equilibrium state for the associated shift map. In Section~\ref{finiteness}, we prove that the inducing time is integrable, in order to project the equilibrium state for the shift map to the original system. The Abramov's formula guarantees that the projected probability is an equilibrium state. 
	We prove that there exist finitely many ergodic equilibrium states that are expanding measures by using the fact that there exists at most a unique equilibrium state for each inducing scheme and that they are finitely many in the Markov structure.
	
	In the last section, we give a class of potentials that we call \text{expanding} and show that this class is equivalent to the class of hyperbolic potentials. This class was inspired by the work \cite{PV}, where the authors consider a type of potentials that they call expanding and they announced existence and uniqueness of measure of maximal entropy for Viana maps.  Moreover, by using a result for Viana maps in \cite{ALP}, where the authors announced the existence of countably many ergodic measures of maximal entropy, we show that the null potential is expanding if the $SRB$ measure is not a measure of maximal entropy, and so is hyperbolic. Also, it implies that the potentials whose integral with respect to every invariant measure vanishes are expanding (and hyperbolic). If the SRB measure is a measure of maximal entropy, we  can still use the strategy of the proof of Theorem \ref{A} to obtain finitely many ergodic measures of maximal entropy. This approach shows the main novelty of our work when comparing to \cite{RV}, because we give an example of setup with critical set where our theory can be applied.
	
	
	
	From now on, consider  $M$ a connected compact metric space and $f:M \to M$ a continuous map. We say that $x\in M$ is a \emph{critical point} for $f$ if there is no neighbourhood $V$ of $x$ in $M$ such that the restriction of $f$ to $V$ is a homeomorphism onto its image.

	\subsection{Topological pressure}
	
	We recall the definition of relative pressure for non-compact sets by means of dynamical balls, for a potential  $\phi: M \to \mathbb{R}$. Given $\delta > 0$, $n \in \mathbb{N}$ and $x \in M$, we define the 
	{\textit{dynamical ball}} $B_{\delta}(x,n)$ as 
	$$
	B_{\delta}(x,n) = \{y \in M \mid d(f^{i}(x),f^{i}(y)) < \delta, \,\, \text{for} \,\, 0 \leq i \leq n\}.
	$$
	Consider for each $N \in \mathbb{N}$, the set
	$$
	\mathcal{F}_{N} = \left\{B_{\delta}(x,n) \mid x \in M\text{ and } n \geq N\right\}.
	$$
	Given an $f$-invariant set $\Lambda \subset M$, not necessarily compact, denote by $\mathcal{F}_{N}(\Lambda)$ the finite or countable families of elements   in $ \mathcal{F}_{N}$ 
	that cover $\Lambda$.
	Define for  $x\in M$ and   $n \in \mathbb{N}$
	$$
	S_{n}\phi(x) = \phi(x) + \phi(f(x)) + \dots + \phi(f^{n-1}(x)),
	$$
	and
	$$
	\displaystyle R_{n,\delta}\phi(x) = \sup_{y \in B_{\delta}(x,n)} S_{n}\phi(y).
	$$
	Define for each $\gamma \in\mathbb R$
	$$
	\displaystyle m_{f}(\phi, \Lambda, \delta,   \gamma, N) = \inf_{\mathcal{U} \in \mathcal{F}_{N}(\Lambda)} \left\{ \sum_{B_{\delta}(x,n) \in \mathcal{U}}
	e^{-\gamma n + R_{n,\delta}\phi(x)} \right\}.
	$$
	Define also
	$$
	\displaystyle m_{f}(\phi, \Lambda, \delta, \gamma) = \lim_{N \to + \infty} m_{f}(\phi, \Lambda, \delta,   \gamma,N)
	$$
	and
	$$
	P_{f}(\phi, \Lambda, \delta) = \inf \{\gamma \in\mathbb R \mid m_{f}(\phi, \Lambda, \delta, \gamma) = 0\}.
	$$
	Finally, define the {\textit{relative pressure}} of $\phi$ on $\Lambda$ as
	$$
	P_{f}(\phi,\Lambda) = \lim_{\delta \to 0} P_{f}(\phi, \Lambda, \delta).
	$$
	The {\textit{topological pressure}} of $\phi$ is, by definition, $P_{f}(\phi) = P_{f}(\phi, M)$. It satisfies
	\begin{eqnarray}
		\label{Pressures}
		P_{f}(\phi) = \sup \{P_{f}(\phi,\Lambda), P_{f}(\phi,\Lambda^{c})\},
	\end{eqnarray}
	where $\Lambda^{c}$ denotes the complement of $\Lambda$ on $M$. We refer the reader to \cite[Theorem 11.2]{Pe2}   for the proof of~\eqref{Pressures} and for additional properties of the pressure.
	
	\subsection{Expanding measures and hyperbolic potentials}


	Given $0<\sigma<1$ and $ \varepsilon > 0$,
	we say that $n \in \mathbb{N}$ is a $(\sigma, \varepsilon)$-{\textit{hyperbolic time}} for $x\in M$ if
	\begin{itemize}
		\item there exists a neighbourhood $V_{n}(x)$ of $x$ such that $f^{n}$ maps  $ V_{n}(x)$ homeomorphically onto the ball $ B_{\varepsilon}(f^{n}(x))$;
		\item $d(f^{i}(y),f^{i}(z)) \leq \sigma^{n-i}d(f^{n}(y),f^{n}(z))$, for all $y,z \in V_{n}(x)$ and all $0 \leq i \leq n-1.$
	\end{itemize}
	We say that $x \in M$ has {\textit{positive frequency}} of $(\sigma,\varepsilon)$-hyperbolic times if
	$$
	d_{\sigma,\varepsilon}(x):=\limsup_{n \to \infty}\frac{1}{n}\#\{0 \leq j \leq n-1 \,\, | \,\, j \,\, \text{is a} \,\, (\sigma,\varepsilon)\text{-hyperbolic time for} \,\, x\} > 0,
	$$
	and define the {\textit{expanding set}}
	$$
	H (\sigma,\varepsilon) = \{x \in M \mid  d_{\sigma,\varepsilon}(x)>0\}.
	$$
	We say that a Borel probability  measure $\mu$ (not necessarily $f$-invariant) on $M$   is {\textit{expanding}} if $\mu(H (\sigma,\varepsilon))=1$, and that a   continuous function $\phi : M \to \mathbb{R}$ is a \emph{$(\sigma,\varepsilon)$-hyperbolic potential} if   the topological pressure $P_{f}(\phi)$ is located on $H(\sigma,\varepsilon)$, i.e.
	$$P_{f}(\phi,H(\sigma,\varepsilon)^{c}) < P_{f}(\phi).$$
	%
	%
	%
	%
	%
	Now, we state our first main result on the existence and finiteness of equilibrium states.
	\begin{maintheorem}
		\label{A}
		Let $f:M \to M$ be a continuous map  and   $\phi:M \to \mathbb{R}$ be a H\"older continuous potential with $P_{f}(\phi)<\infty $. If there exist $0<\sigma<1$ and $\varepsilon>0$ such that $\phi$ is  a $(\sigma,\varepsilon)$-hyperbolic potential,  then
		there exists some ergodic equilibrium state for $\phi$. Moreover, the number of  these  equilibrium states is finite and they are all expanding measures. 
	\end{maintheorem}
	
	 \subsection{Expanding potentials}\label{expanding} Let $\phi:M \to \mathbb{R}$ be a continuous function. We say that $\phi$ is a $(\sigma,\epsilon)$-\textit{expanding potential} if the following inequality holds
	\[
	\displaystyle \sup_{\mu \in \mathcal{E}^{c}} \bigg{\{} h_{\mu}(f) + \int \phi d \mu \bigg{\}} < \sup_{\mu \in \mathcal{E}} \bigg{\{} h_{\mu}(f) + \int \phi d \mu \bigg{\}} = P_{f}(\phi),	
	\]
	where $\mathcal{E}$ denotes the set of all expanding measures for $f$. It means that the topological pressure $P_{f}(\phi)$ is concentrated at the expanding measures.

Our second main theorem establishes an equivalence between the class of $(\sigma,\epsilon)$-hyperbolic potentials and the class of $(\sigma,\epsilon)$-expanding potential.
	
	 \begin{maintheorem}\label{B}	Let $f:M \to M$ be a continuous map  and $\phi : M \to \mathbb{R}$ a continuous potential. Then $\phi$ is a $(\sigma,\epsilon)$-hyperbolic potential if, and only if, $phi$  is an $(\sigma,\epsilon)$-expanding potential. In particular, if $\phi$ is $(\sigma,\epsilon)$-hyperbolic and H\"older with finite pressure $P_{f}(\phi)$, there exist finitely many equilibrium states which are expanding measures.
	\end{maintheorem}
 
 \subsection{Viana maps} We apply Theorems \ref{A} and \ref{B}  to an  open  class of   nonuniformly expanding maps with critical points  defined in the two-dimensional cylinder $S^1\times \mathbb R$. These maps have been   
introduced by Viana in~\cite{V} and have natural extensions to  higher dimensions; see \cite[Section 2.5]{V}. 

 Let $a_0\in(1,2)$ be a parameter close to 2 such that the critical point $x=0$ is pre-periodic for the quadratic map $q(x)=a_0-x^2.$  Consider the circle $S^1=\mathbb R/\mathbb Z$ and $b:S^1\rightarrow \mathbb R$ a Morse function, for instance, $b(s)=\sin(2\pi s)$. For small $\alpha>0$, define
\begin{equation}\label{eq.viana}
	\begin{array}{rccc} \hat f: & S^1\times\mathbb R,
		&\longrightarrow & S^1\times \mathbb R\\
		& (s, x) &\longmapsto & \big(\hat g(s),\hat q(s,x)\big),
	\end{array}
\end{equation}
 where $\hat g:S^1\to S^1$ is the uniformly expanding map defined, for some $d\ge16$, by 
$\hat{g}(s)=ds$ (mod 1) and $$\hat q(s,x)=a(s)-x^2,$$ with 
$$a(s)=a_0+\alpha b(s).$$
It is easily verified  that, for any small $\alpha>0$, there exists an interval $I\subset (-2,2)$ for which $\hat f(S^1\times I)$ is contained in the interior of $S^1\times I$. This implies that  any   $f$ sufficiently close to $\hat f$ in the $C^0$ topology still has $S^1\times I$  as a forward invariant region. 
We define the family~$\mathcal V$ of \emph{Viana maps}\index{Viana!map} as the set of $C^3$ maps in a sufficiently small neighbourhood of $\hat f:S^1\times I\to S^1\times I$ in the  $C^2$ topology.

From the expression in~\eqref{eq.viana} we easily see that $\{x=0\}$ is a critical set for $\hat f$. By an implicit function argument we can easily see that a critical curve close to $\{x=0\}$ still exists for $f\in\mathcal V$. The skew-product expression in~\eqref{eq.viana}  gives that the vertical lines form an invariant 
for $\hat f$.  An invariant foliation by nearly vertical lines still exists for any $f\in\mathcal V$; see 	\cite[Section 2.5]{V}. For any $f\in\mathcal V$, the dynamics in the space of these invariant foliations is driven by an expanding map in the circle with a fixed point $p_f\in S^1$. This gives rise to a nearly vertical invariant leave on which the dynamics is close to the quadratic map $Q$, and so it has a repelling fixed point $q_f\in I$. Clearly,  $\Lambda_f=\{(p_f,q_f)\}$ is a  repelling fixed point for $f$.
Our third main theorem is an application of Theorems \ref{A} and \ref{B} together. 
    \begin{maintheorem}\label{C}
	Let $f:S^{1} \times I \to S^{1} \times I$ be a Viana map. If the unique $SRB$ measure is not a measure of maximal entropy, we have that the null potential is expanding and, by Theorem~\ref{B} it is hyperbolic. Moreover, by Theorem \ref{A} we have existence of some and at most finitely many ergodic measures of maximal entropy,  all expanding measures. 
   \end{maintheorem}

	\section{Equilibrium states for   lifted dynamics}\label{Lifted}
	
	In this section we begin the proof of Theorem A. It will be divided into several steps. The strategy is to lift the dynamics to a Markov structure,
	finding equilibrium states for the induced potential and then projecting them.

	\subsection{Inducing schemes} An inducing scheme for   $f:M\to M$ on an open set $U$ with inducing time $\tau: U \rightarrow \mathbb{N}$ is a pair $(F,\mathcal{P})$, where
	\begin{itemize}
		\item $\mathcal{P} = \{P_{1}, P_{n},\dots\}$ is a partition by measurable subsets of $U$ with nonempty interior 
		such that $\tau|_{Q_i}\equiv \tau_i$ is constant, for all $ i\ge1$; 
	    \item   $F:U \to U$ is a map defined by $F(x)=f^{\tau(x)}(x)$ is a homeomorphism from the interior of each $Q_i$ 
	     onto~$U$.
	\end{itemize}
	Given an inducing scheme $(F,\mathcal{P})$ and an ergodic $f$-invariant probability $\mu$, we say that $\mu$ is \emph{liftable} to $(F, \mathcal{P})$ if there exists an $F$-invariant finite measure 
	$\overline{\mu} \ll \mu$ such that for every measurable set $A \subset M$, 
	$$
	\mu(A) = \sum_{k=1}^{\infty}
	\sum_{j=0}^{\tau_{k} - 1} \overline{\mu}(f^{-j}(A) \cap P_{k}).
	$$	
	Inversely, given an $F$-invariant probability $\overline{\eta}$ we define its \textit{projection} $\eta$ by  
	\begin{equation}\label{eq.projecao}
	\eta(A) = \sum_{k=1}^{\infty}
	\sum_{j=0}^{\tau_{k} - 1} \overline{\eta}(f^{-j}(A) \cap P_{k}).
	\end{equation}
	It easy to see that $\eta$ is a $f$-invariant probability measure and $\eta(M) = \int \tau d\overline{\eta}$. Given a potential $\phi : M \to \mathbb{R}$ we define the \emph{induced potential} $\overline{\phi}$ by 
	$$
	\overline{\phi}(x) =\sum_{j=0}^{\tau(x)-1} \phi(f^{j}(x)).
	$$	
	The next result establishes \emph{Abramov's Formulas}.
	
	\begin{proposition}[Zweim\" uller \cite{Z}]\label{pr.zwei} If $\mu$  is liftable to $\overline{\mu}$, then  
		$$
		h_{\mu}(f) = \frac{h_{\overline{\mu}}(F)}{\int \tau d \overline{\mu}} \quad\text{and}\quad \int \phi d \mu = \frac{\int \overline{\phi} d\overline{\mu}}{\int \tau d \overline{\mu}}.
		$$
	\end{proposition}

	We say that an inducing scheme $(F,\mathcal{P})$ defined on an open set $Y$ is \textit{compatible} with a measure $\mu$ if it holds that
	\begin{itemize}
		\item $\mu(Y) > 0$;
		
		\item $F$ is non-singular (that is, $F_{*}\mu \ll \mu$ where $F_{*}\mu = \mu \circ F^{-1}$);
		
		\item $\mu(\cup_{P \in \mathcal{P}} P) = \mu(Y)$.
	\end{itemize}
	Our goal is to explore the established dictionary between the equilibrium states of $\phi$ and the equilibrium states of $\overline{\phi}$. The next result  assures that every ergodic expanding measure can be lifted to some inducing scheme.
	
	\begin{theorem}[Pinheiro \text{\cite[Theorems 1 and D]{Pi1}}] 
		\label{Pinheiro}
		There exist open sets $U_1,\dots, U_s$ and $K>0$ such that, for each $1\le i\le s$, there exists an  inducing scheme  $(F_i,\mathcal{P}_i)$  on $U_i$ with induced time $\tau_i: U_i \to \mathbb{N}$ and, given any ergodic probability measure $\mu$ with $\mu(H(\sigma,\epsilon))=1$,
			\begin{enumerate}
			 \item  there exists $1\leq i\leq s$ such that $\mu$ is liftable to some  $F_i$-invariant measure $\overline{\mu}_i$;
			 \item $\int \tau_i d\,\overline{\mu}_i \le K$;
			 \item $(F_{i},\mathcal{P}_{i})$ is compatible with $\mu_{i}$.
		   \end{enumerate}
	\end{theorem}
	
	\begin{remark}\label{U_infty}
		When an inducing scheme $(F,\mathcal{P})$ is compatible with an invariant measure $\mu$, we have that $\overline{\mu}$ is also invariant and it holds that
		\[
		\overline{\mu}\bigg{(}\bigcap_{i=0}^{\infty} F^{-i}\bigg{(}\bigcup_{P \in \mathcal{P}} P \bigg{)}\bigg{)} = \overline{\mu}\bigg{(}\bigcup_{P \in \mathcal{P}} P \bigg{)} = \overline{\mu}(Y)
		\]
	\end{remark}
	
	\subsection{Hyperbolic potentials}
	Given an invariant measure~$\mu$ and a potential $\phi: M \to \mathbb{R}$, we say that the sum
	\[
	h_{\mu}(f) + \int \phi d\mu
	\]
is the \textit{free energy} or \textit{pressure of the measure} of the system $(f,\mu,\phi)$.
The next proposition and Theorem \ref{Pinheiro} guarantee that every ergodic expanding measure with \textit{``high free energy''} and, in particular, those which are candidates 
to be equilibrium states, can be lifted to some inducing scheme.

\begin{proposition}\label{pr.hypmeas}
	Let $\phi$ be  a $(\sigma,\epsilon)$-hyperbolic potential. If $\mu$ is an ergodic probability measure such that $h_{\mu}(f) + \int \phi d\mu > P_{f}(\phi,H(\sigma,\epsilon)^{c})$, then $\mu(H(\sigma,\epsilon))=1$.
\end{proposition}

\begin{proof}
	Since $H(\sigma,\epsilon)$ is an invariant set and $\mu$ is an ergodic probability measure, we have $\mu(H(\sigma,\epsilon)) = 0$ or $\mu(H(\sigma,\epsilon)) = 1$. Since the potential $\phi$ is $(\sigma,\epsilon)$-hyperbolic, we get
	$$
	h_{\mu}(f) + \int \phi d\mu > P_{f}(\phi,H(\sigma,\epsilon)^{c}) \geq \sup_{\nu(H(\sigma,\epsilon)^{c})=1}
	\bigg{\{}h_{\nu}(f) + \int_{H(\sigma,\epsilon)^{c}} \phi d\nu\bigg{\}};
	$$
	for the second inequality, see  \cite[Theorem A2.1]{Pe2}.
	So, we cannot have $\mu(H(\sigma,\epsilon)^{c}) = 1$ and we obtain $\mu(H(\sigma,\epsilon)) = 1$, which means that $\mu$ is an expanding measure.
\end{proof}

\begin{corollary}\label{fixed}
Given a $(\sigma,\epsilon)$-hyperbolic potential $\phi$, there exist an inducing scheme $(F,\mathcal P)$ and a sequence $(\mu_n)_n$ of ergodic probability measures  liftable to $(F,\mathcal P)$  such that  $$h_{\mu_n}(f) + \int \phi d\mu_n\to P_f(\phi).$$
\end{corollary}
\begin{proof}
	We can choose a sequence of ergodic probabilities $(\mu_n)_n$ such that
	$$h_{\mu_n}(f) + \int \phi d\mu_n\to P_f(\phi).$$
	Since Theorem \ref{Pinheiro} guarantees finitely many inducing schemes for which these measures are liftable and by Proposition \ref{pr.hypmeas} there exists $n_{0}$ such that $\mu_{n}(H(\sigma,\epsilon)) = 1,$ for all $ n \geq n_{0}$, by choosing a subsequence if necessary, we can find an inducing scheme $(F,\mathcal{P})$ with respect to which the measures are liftable. 
\end{proof}

From now and on, we fix an inducing scheme $(F,\mathcal{P})$ given by Corollary \ref{fixed}.

	\subsection{Thermodynamic formalism on inducing schemes} \label{MarkovShifts} 
	Here we code the dynamics of  the inducing scheme $(F,\mathcal{P})$  using  a countable Markov shift and present some useful  results,  most of them  due to Sarig.
	
 Consider the space of symbols $S=(P_i)_{i\geq 1}$, where $P_1,P_2,\dots$ are the elements of the partition $\mathcal P$, and the full shift  $\sigma: S^{\mathbb{N}} \to S^{\mathbb{N}}$.  We  define  $U_\infty \subset U$ as the subset of points $x$ such that $F^n(x)$ is defined for every $n\geq 0$,~i.e.

	\[
	\displaystyle U_\infty = \bigcap_{k=0}^{\infty} F^{-k}\bigg{(}\bigcup_{i=1}^{\infty}P_{i}\bigg{)}.
	\]	
	By Remark \ref{U_infty}, we have  $\mu(U_{\infty}) = 1$, for every measure $\mu$ such that $(F,\mathcal{P})$ is compatible with~$\mu$. 
	Since $F$ expands distances at a uniform rate, it is easy to check that  $\pi: \Sigma \rightarrow U_\infty$, defined for  $\overline{a}=(P_{a_0},P_{a_1},\dots)$
	by  
	$$
	\pi(\overline{a})=\bigcap_{i\geq 0} F^{-i}(P_{a_i}),
	$$  is a topological conjugacy. Observe that $\pi$   preserves ergodicity and entropy between invariant measures of $F$ and invariant measures of $\sigma$. Indeed, we may consider the map 
	\begin{eqnarray*}
		\pi^{*}: \mathcal{M}_{F}(U_{\infty})& \to & \mathcal{M}_{\sigma}(\Sigma) \\
		\overline{\mu} & \to & \overline{\mu}\circ \pi
	\end{eqnarray*}
	defined by $\pi^{*}(\overline{\mu})(A) = \overline{\mu}(\pi(A))$. We can also define its inverse by 
	\begin{eqnarray*}
		\pi_{*}: \mathcal{M}_{\sigma}(\Sigma)& \to & \mathcal{M}_{F}(U_{\infty}) \\
		\eta & \to & \eta \circ \pi^{-1}
	\end{eqnarray*}
	where $\pi_{*}(\eta)(B) = \eta(\pi^{-1}(B))$. We define an $F$-\textit{cylinder} as a set $C_n$ of the form
	$$C_{n} = P_{i_0}\cap F^{-1}(P_{i_1})\cap F^{-2}(P_{i_2})\cap \dots \cap F^{-n}(P_{i_{n-1}}).
	$$ 
	Observe that $\pi^{-1}(C_n)$ is a cylinder on $\Sigma$ in the sense of \cite[Section 2]{Sa1}. Thus, given an induced potential $\overline{\phi}$ we define a potential on $\Sigma$ by 
	$$
	\Phi\big(\overline{a}\big)=\overline{\phi}\big(\pi(\overline{a})\big).
	$$		
We say that a probability measure $m_\Phi$  (resp. $m_{\overline\phi}$) on $\Sigma$ (resp. $U_\infty$) is	
$\Phi$-\textit{conformal} (resp. $\overline\phi$-\textit{conformal}) if (see \cite{Sa4}, Definition 2.6)
\[
\frac{d m_{\Phi}}{d m_{\Phi} \circ \sigma} = e^{\Phi}, \,\, \Bigg{(}resp. \,\, \frac{d m_{\overline{\phi}}}{d m_{\overline{\phi}} \circ F} = e^{\overline{\phi}} \Bigg{)}.
\]
As a consequence, we have $m_{\Phi}(\sigma(A)) = \int_{A} e^{-\Phi} dm_{\Phi}$ (resp. $m_{\overline\phi}(F(A)) = \int_{A} e^{-\overline\phi} dm_{\overline\phi}$), whenever $\sigma\vert_A$ (resp. $F\vert_A$) is injective. Note that 
 $$m_{\Phi} \text{ is 	
$\Phi$-conformal } \iff m_{\overline\phi} \text{ is $\overline\phi$-conformal. }$$
We say that an invariant probability measure $\mu_\Phi$  (resp. $\mu_{\overline\phi}$) on $\Sigma$ (resp. $U_\infty$) is a	
\textit{Gibbs measure} if for some global constants $K,P$ and every cylinder $\tilde{C}_{n}$ the following holds (see \cite{Sa2}, Introduction) 
\[
K^{-1} \leq \frac{\mu_{\Phi}(\tilde{C}_{n})}{e^{-nP + S_{n}\Phi(x)}} \leq K, \,\, \text{for all} \,\, x \in \tilde{C}_{n}.
\]	
Respectively, for every $F$-cylinder $C_{n}$ and some global constants $K,P$ the following holds
\[
K^{-1} \leq \frac{\mu_{\overline{\phi}}(C_{n})}{e^{-nP + S_{n}\overline{\phi}(x)}} \leq K, \,\, \text{for all} \,\, x \in C_{n}.
\]	 	
The topological conjugacy  $\pi$ allows us to translate some  known  results  for countable Markov  shifts to our context of inducing schemes. A list of those results will be given below. Note that, since we are dealing with full branch inducing schemes, the  shift  $\sigma$ is   topologically mixing. We define the \emph{$n$-th variation} of $\overline{\phi}$ as
	$$
	V_{n}(\overline{\phi}) : = \sup\left\{|\overline{\phi}(x) - \overline{\phi}(y)| \colon  x, y \in C_{n}\right\},
	$$ where the supremum is taken over all cylinders $C_n$. We say that a potential
	$\overline{\phi} : U_{\infty} \to \mathbb{R}$ has \emph{summable variation} if  $\sum_{n\ge1} V_{n}(\overline{\phi}) < \infty$.  The potential $\overline{\phi}$  is called \emph{locally H\"older} if there exist $A>0$ and $0<\theta<1$ such that $V_n(\overline{\phi}) \leq A \theta^n$, for all $n\geq 1$.  Observe that any locally H\"older potential has summable variation.  
	
\begin{proposition}\label{pr.todd}
		If $\phi: M \to \mathbb{R}$ is a H\"older potential, then  $\overline{\phi} : U_{\infty} \to \mathbb{R}$ is a locally H\"older potential.
	\end{proposition}
	
	\begin{proof}
		Since $\phi$ is H\"older, there are constants $\rho, \alpha > 0$ such that $| \phi(x) - \phi(y) | \leq \rho d(x,y)^{\alpha}$. We must show that there are   
		$A > 0$ and $\theta \in (0,1)$ such that 
		\begin{equation}\label{goal}
			|\overline{\phi}(x) - \overline{\phi}(y)| \leq A \theta^{n}, \quad\forall x,y \in C_{n}.
		\end{equation}
		In fact, given $x,y \in C_{n}$, 
		there are $P_{i_{0}}, P_{i_{1}}, \dots, P_{i_{n-1}}$ such that $F^{k}(x), F^{k}(y) \in P_{i_{k}}$. Then,  
		\begin{eqnarray*}
			|\overline{\phi}(x) - \overline{\phi}(y) | &= &
			\left|
			\sum_{j=0}^{\tau_{i_{0}}-1}
			\phi(f^{j}(x)) - \phi(f^{j}(y))\right|
			\leq 
			\sum_{j=0}^{\tau_{i_{0}}-1}
			\left|\phi(f^{j}(x)) - \phi(f^{j}(y))\right|\\
			&\leq &
			\rho \sum_{j=0}^{\tau_{i_{0}} - 1} d(f^{j}(x),f^{j}(y))^{\alpha}
			\leq  \rho \sum_{j=0}^{\tau_{i_{0}} - 1} (\sigma^{\tau_{i_{0}} - j}d(F(x),F(y))^{\alpha}\\
			&\leq &
			\rho \sum_{j=1}^{\tau_{i_{0}}} \sigma^{\alpha j} d(F(x),F(y))^{\alpha}
			\leq 
			\rho \sum_{j=1}^{\tau_{i_{0}}} \sigma^{\alpha j}\sigma^{\tau_{i_{1}}\alpha} d(F^{2}(x),F^{2}(y))^{\alpha}\\ 
			&\leq &
			\rho \sum_{j=1}^{\tau_{i_{0}}} \sigma^{\alpha j}\sigma^{\tau_{i_{1}}\alpha} \dots \sigma^{\tau_{i_{n-1}}\alpha} d(F^{n}(x),F^{n}(y))^{\alpha} 
			\leq 
			\rho \sum_{j=1}^{\tau_{i_{0}}} \sigma^{\alpha j} \sigma^{(n-1) \alpha} (2 \delta)^{\alpha}\\
			&\leq &
			\rho (2 \delta)^{\alpha} \sum_{j=0}^{\infty} \sigma^{\alpha j} \sigma^{n \alpha}.\\  
		\end{eqnarray*}
		This yields~\eqref{goal}  with $A  = \rho (2\delta)^{\alpha} (\sum_{j=0}^{\infty} (\sigma^{\alpha})^{j})$ and $\theta  = \sigma^{\alpha}$.
	\end{proof}

The   \emph{Gurevich pressure}  is defined as
	$$
	\displaystyle P_{G}(\overline{\phi},a)= \lim_{n \to \infty} \frac{1}{n} \log \left(\sum_{F^{n}(x) = x, x_{0} = a} e^{\overline{\phi}_{n}(x)} \right),
	$$
	where $\overline{\phi}_{n}(x) = \sum_{j=0}^{n - 1} \overline{\phi}(f^{j}(x)).$ It follows from   \cite[Theorem 1]{Sa1}  that $P_{G}(\overline{\phi},a)$ does not depend on $a$ and we denote it by $P_{G}(\overline{\phi})$.

	\begin{theorem}[Sarig \text{\cite[Theorem 2]{Sa1}}]\label{Sarig} If $(\Sigma, \sigma)$ a countable full shift and $ {\Phi}$ is locally H\"older, then the Gurevich Pressure is well defined and
		independent of $a$. 
		Moreover,
		$$
		P_{G}(\Phi) = \sup \left\{ P_{\topo}( {\Phi}_{|Y}) \mid Y \subset \Sigma \,\, \text{is a topologically mixing finite Markov shift}\right\},
		$$
		where $P_{\topo}({\Phi}_{|Y})$ is the topological pressure of the restriction of $ \Phi$ to   $Y$.
	\end{theorem}
	

	\begin{theorem}[Iommi-Jordan-Todd \text{\cite[Theorem 2.1]{IJT}}] \label{VarPrinc} 
	Let $(\Sigma,\sigma)$ be a countable full shift and 
		$\Phi : \Sigma \to \mathbb{R}$ a potential with summable variation and $\sup \Phi < \infty$. Then
		$$
		P_{G}(\Phi) = \sup \bigg{\{}h_{\nu}(F) + \int \Phi d \nu \mid -\int\Phi d\nu < \infty\bigg{\}}.
		$$
	\end{theorem}


Let $T=(t_{ij})$ be the matrix transition of the shift $(\Sigma,\sigma)$, where $\Sigma = S^{\mathbb{N}}$. We say that it  has the \emph{big images and preimages (BIP) property} if 
$$
\exists \,\, b_{1},\dots,b_{N} \in S \,\,
\forall a \in S \,\, \exists \,\, i,j \in \{1,\dots,N\} \,\, \text{:} \,\, t_{b_{i}a}t_{ab_{j}} = 1.
$$
Clearly, if $(\Sigma,\sigma)$ is a full shift, then it has the BIP property. 
Given $a \in S$, let $[a] = \{(x_{0},x_{1},\dots) \mid x_{0} = a\}$ and $\Phi_{a}(x) = 1_{[a]} \inf \{n \geq 1 \mid \sigma^{n}(x) \in [a]\}$. Set
\[
Z_{n}(\Phi,a) := \sum_{\sigma^{n}(x) = x} e^{\Phi_{n}(x)}1_{[a]}(x), \quad Z_{n}^{*}(\Phi,a) := \sum_{\sigma^{n}(x) = x} e^{\Phi_{n}(x)}1_{[\Phi_{a}=n]}(x) 
\]
and  $$\lambda : = e^{P_{G}(\Phi)}.$$ We say that a potential $\Phi$ is \textit{recurrent} if for some $a \in S$ we have that $\sum_{n \geq 1} \lambda^{-n}Z_{n}(\Phi,a)$ diverges and it is \textit{positive recurrent} if it is recurrent and $\sum_{n \geq 1} n\lambda^{-n}Z_{n}^{*}(\Phi,a) < \infty$. 

\begin{theorem}\label{SarigBIP}\cite[Theorem 1]{Sa2} If $(\Sigma; \sigma)$ is topologically mixing and $\sum_{n\geq 1} V_n(\Phi) <\infty$, then $\Phi$ has
an invariant Gibbs measure if, and only if, it has the BIP property and $\Phi$ has finite Gurevich pressure.
\end{theorem}

\begin{corollary}\label{SarigRecurrent}\cite[Corollaries 1 and  2]{Sa2} If $(\Sigma,\sigma)$ has the BIP property, then
\[
\displaystyle P_{G}(\Phi)= \lim_{n \to \infty} \frac{1}{n} \log \left(\sum_{\sigma^{n}(x) = x, x_{0} = a} e^{\Phi_{n}(x)} \right)
\]
and $\Phi$ is positive recurrent.	
\end{corollary}

In the next theorem we give  a list of results due to Sarig  on the existence and uniqueness of conformal measures, Gibbs measures and equilibrium states for countable Markov shifts; see  \cite{Sa1,Sa2,Sa3, Sa4}. 


\begin{theorem}\label{Sarigseveral}	Let $(\Sigma,\sigma)$ be a full shift and  $\Phi: \Sigma \to \mathbb{R}$ a positive recurrent potential with  $\sum _{n\geq 1} V_n(\Phi) <\infty$  and $P_G(\Phi)=0$. If $L$ is the operator given by
	$$
	L (\Psi)(x) = \sum_{\sigma(y)=x} e^{\Phi(y)}\Psi(y),
	$$ 
	then
	\begin{enumerate}
		
		\item \cite[Theorem 2 \& Corollary 2]{Sa2} there is a finite measure $m_\Phi$ such that ${L^{*}(m_{\Phi}) =  m_{\Phi}}$; this means that $m_\Phi$ is $\Phi$-conformal; 
		
		\item \cite[Theorem 2 \& Corollary 2]{Sa2} there is a positive continuous function $h_\Phi$ such that $L(h_{\Phi}) = h_{\Phi}$;
		
		\item \cite[Theorem 2]{Sa2} if $\mu_{\Phi}$ is defined by  $d \mu_{\Phi} = h_{\Phi} d m_{\Phi}$, then $\mu_\Phi$ is an invariant probability measure and it is a Gibbs measure;

		\item \cite[Corollary 2]{Sa2} if either $h_{\mu_{\Phi}} (\sigma) < \infty$ or $-\int \Phi d \mu_{\Phi} < \infty$,
		then $\mu_{\Phi}$ is the unique equilibrium state; in particular, $P_{G}(\Phi) = h_{\mu_{\Phi}}(\sigma) + \int \Phi d\mu_{\Phi}$=0;

		\item \cite[Remark 3]{Sa1} $h_{\Phi}$ is unique and $m_{\Phi}$ is the unique $\Phi$-conformal probability measure ($h_{\Phi}$ is uniquely determined
		up to a multiplicative constant).
	\end{enumerate}

\end{theorem}

\subsection{Zero pressure and equilibrium states}\label{zero}	
The main goal of this subsection is to obtain Proposition~\ref{Zero}, where we  show  that hyperbolic potential with zero pressure lift to potentials with zero Gurevich pressure.

	Given  a hyperbolic potential $\phi : M \to \mathbb{R}$, note that if we set $\varphi : = \phi - P_{f}(\phi)$, then $\varphi$ is   a hyperbolic potential with $P_{f}(\varphi) = 0$. It is easy to see that the equilibrium states for the potentials $\phi$ and $\varphi$ are exactly the same (if any exists).
	
	\begin{proposition}\label{FiniteSup}
		If  $\varphi$ is a H\"older potential  such that $P_{f}(\varphi)=0$, then   $\sup \overline{\varphi}< \infty$.
	\end{proposition}
	
	\begin{proof}
		Since $P_{f}(\varphi)=0$, the Variational Principle for compact sets (see \cite{OV},\cite{W}) implies that
		\[
		0 = \sup_{\eta \in \mathcal{M}_{f}(M)} \bigg{\{} h_{\eta}(f) + \int \varphi d\eta \bigg{\}}.
		\]
		Once $h_{\eta}(f) \geq 0$ for every probability $\eta \in \mathcal{M}_{f}(M)$, it means that $\int \varphi d\eta \leq 0$ for every probability $\eta \in \mathcal{M}_{f}(M)$. The Abramov's Formulas imply that $\int \overline{\varphi} d\overline{\eta} \leq 0$ for every probability $\eta \in \mathcal{M}_{f}(M)$. 
		
		We claim that in the inducing scheme $(F,\mathcal{P})$ that we are considering, the partition $\mathcal{P} = \{P_{1},\dots,P_{n},\dots\}$ is such that each element $P_{i}$ has a point $x_{i}$ such that $\overline{\varphi}(x_{i}) \leq 0$. In fact, given $P_{i} \in \mathcal{P}$, there exists $x_{i} \in P_{i}$ such that $F(x_{i}) = x_{i}$. It means that $x_{i} \in M$ is a periodic point for the map $f$ with period $\tau(x_{i})$. For the probability
		\[
		\eta_{i} := \frac{1}{\tau(x_{i})}\sum_{j=0}^{\tau(x_{i}) - 1}\delta_{f^{j}(x_{i})}
		\]
		we have that $\overline{\eta}_{i} = \frac{1}{\tau(x_{i})}\delta_{x_{i}}$ and $\int \overline{\varphi} d\overline{\eta}_{i} \leq 0$ implies that $\overline{\varphi}(x_{i}) \leq 0$.
		Proposition \ref{pr.todd} guarantees that $\overline{\varphi}$ is locally H\"older, which implies that
		\[
		V_{1}(\overline{\varphi}) : = \sup\left\{|\overline{\varphi}(x) - \overline{\varphi}(y)| \colon  x, y \in P_{i}\right\} \leq A\theta.
		\]
		But we obtain $\overline{\varphi}(x)\leq \overline{\varphi}(x) -\overline{\varphi}(x_{i}) \leq A\theta$, for every $x \in P_{i}$, for every $i \in \mathbb{N}$. It means that $\sup \overline{\varphi} \leq A\theta$.
	\end{proof}

	\begin{proposition}\label{Zero}
		If  $\varphi$ is a H\"older $(\sigma,\epsilon)$-hyperbolic potential  such that $P_{f}(\varphi)=0$, then   ${P_{G}(\overline{\varphi}) = 0}$.
	\end{proposition}
	
	\begin{proof}
		Firstly, we prove that $P_{G}(\overline{\varphi}) \geq 0$. By Proposition \ref{FiniteSup} we have  $\sup \overline{\varphi} < \infty$ and by Theorem \ref{VarPrinc} we have the Variational Principle for the Gurevich pressure. By  Corollary \ref{fixed}, there exists an inducing scheme $(F,\mathcal{P})$ for which there exists a sequence $(\mu_{n})_n$ of expanding ergodic  measures on $M$, liftable to $(F,\mathcal{P})$ such that $h_{\mu_{n}}(f) + \int \varphi d\mu_{n} \to P_{f}(\varphi)=0$. It follows from  Abramov's Formulas (Proposition \ref{pr.zwei})  that
		$$
		h_{\overline{\mu}_{n}}(F) + \int \overline{\varphi} d\overline{\mu}_{n} = \bigg{(}\int \tau d \overline{\mu}_{n}\bigg{)} \bigg{(}h_{\mu_{n}}(f) + \int \varphi d\mu_{n}\bigg{)}.
		$$ 
		Since $h_{\mu_{n}}(f) + \int \varphi d\mu_{n} \to P_{f}(\varphi)=0$ and Theorem \ref{Pinheiro} guarantees that the sequence $\big{\{}\int \tau d \overline{\mu}_{n}\big{\}_{n}}$ is  bounded, we obtain 
		$h_{\overline{\mu}_{n}}(F) + \int \overline{\varphi} d\overline{\mu_{n}} \to 0$. By Proposition \ref{FiniteSup} we have that $\sup \overline{\varphi} < \infty$ and by Theorem \ref{VarPrinc} we have the Variational Principle for the Gurevich Pressure. This gives $  P_G(\overline{\varphi})\ge 0$.
		
		Now, we prove that $P_{G}(\overline{\varphi}) \leq 0$. By taking the finite Markov subshift with symbols $P_{1},P_{2},...,P_{N}$, denoted by $(\Sigma_{N},\sigma_{N})$, we obtain 
		an equilibrium state $\mu_{N}$ for the finite subshift $(\Sigma_{N},\sigma_{N})$ and Theorem~\ref{Sarig}  		 guarantees that the  $F$-invariant measure $\overline{\nu}_{N}=\pi^*(\mu_N)$ is such that $h_{\overline{\nu_{N}}}(F) + \int \overline{\varphi} d\overline{\nu_{N}} = P_{\topo}(\overline{\varphi}_{|\pi(\Sigma_{N})}) 
		\to P_{G}(\overline{\varphi})$. If we had $P_{G}(\overline{\varphi}) > 0$, considering $\nu_N$ the projected $f$-invariant measure as defined by~\eqref{eq.projecao}, for some $N$ big enough we obtain 
		$$
		0 < h_{\overline{\nu}_{N}}(F) + \int \overline{\varphi} d\overline{\nu}_{N} = \bigg{(}\int \tau d \overline{\nu}_{N}\bigg{)} \bigg{(}h_{\nu_{N}}(f) + 
		\int \varphi d\nu_{N}\bigg{)}.
		$$ 
		This implies that
		$$
		0 < h_{\nu_{N}}(f) + \int \varphi d\nu_{N} \leq P_{f}(\varphi) = 0,
		$$
		which gives a contradiction. So, $P_{G}(\overline{\varphi}) \leq 0$.
		Finally, by the Variational Principle (Theorem~\ref{VarPrinc}) for inducing schemes, 
		we obtain $0 \leq P_{\topo}(\overline{\varphi}) = P_{G}(\overline{\varphi}) \leq 0$ and $P_{G}(\overline{\varphi})=0$.
	\end{proof}

Using Proposition \ref{Zero} we are able to use Theorem \ref{Sarigseveral} to obtain a unique $\overline{\varphi}$-conformal probability measure $m_{\overline{\varphi}}$ and a unique invariant Gibbs measure $\mu_{\overline{\varphi}}$ for $F$. 


	\section{Finiteness of ergodic equilibrium states}\label{finiteness}
	
	In this section, we will consider the system $(f,\varphi)$ where $P_{f}(\varphi)= 0$, the lifted system $(F,\overline{\varphi})$ (Proposition \ref{Zero} gives $P_{G}(\overline{\varphi}) = 0$), the $\overline{\varphi}$-conformal measure $m_{\overline{\varphi}}$ and the Gibbs measure $\mu_{\overline{\varphi}}$ (both given by Theorem \ref{Sarigseveral}). Following ideas of Iommi and Todd in \cite{IT2}, we will prove that the inducing time is integrable with respect to the Gibbs measure $\mu_{\overline{\varphi}}$. As a consequence, we have that $-\int \overline{\varphi} d\mu_{\overline{\varphi}} < \infty$ and the Gibbs measure is the unique equilibrium state for the induced potential. Also, we can project it to a measure $\mu_{\varphi}$, which will be   an equilibrium state for the original system $(f,\varphi)$. It can be done for each inducing scheme of the Markov structure for which sequences of measures are liftable, as in the previous section. So, at the end we obtain finitely many equilibrium states because the Markov structure is finite.
	
	We will state results  adapted from \cite{IT2}   whose proofs follow exactly along the same lines.
	The strategy is, firstly, to show that for an inducing scheme given by Corollary \ref{fixed}, a measure with low free energy is far from being a Gibbs measure  (Proposition \ref{Gibbs}). As a consequence, the accumulation point, in the weak* topology, of a sequence of measures with free energy converging to zero (pressure) will provide us a Gibbs measure, which is $\mu_{\overline{\varphi}}$, by uniqueness. This measure is an equilibrium state and, by the Abramov's formulas, this can be projected to an equilibrium state $\mu_{\varphi}$  (Proposition \ref{Finitely}).
	
	\subsection{Measures with low free energy}
	The main goal of this subsection is to prove Proposition~\ref{Gibbs} below, where we obtain a key property for the conformal measure of an inducing scheme.

 Given an inducing scheme $(F,\mathcal{P})$ and $n\ge 1$, 	
let   $\mathcal{P}_{n}$ denote the partition made of cylinders~$C_{n}^{i}$ relative to the inducing scheme $(F,\mathcal{P})$,   i.e.
    \[
    \mathcal{P}_{n}^F = \bigvee_{j=0}^{n-1} F^{-j}(\mathcal{P}).
    \]

	\begin{proposition}
		\label{Gibbs}
		Given a hyperbolic potential $\varphi$ with $P_{f}(\varphi)=0$ and an inducing scheme $(\tilde{F},\tilde{\mathcal{P})}$, there exists $k \in \mathbb{N}$ such for   $(F,\mathcal P)$, with $F=\tilde{F}^{k}$ and $\mathcal P=\mathcal{P}^{\tilde F}_k$, the following holds: denoting $m_{\overline{\varphi}}$ the conformal measure for  $(F,\overline{\varphi})$, there exist  $\gamma_{0} \in (0,1)$ and,  for any $n\in\mathbb N$ and any cylinder $C_{n}^i  \in \mathcal{P}_{n}^{F}$,   a constant  $\delta_{n,i}<0$ such that, for any  measure $\mu_{F} \in \mathcal{M}_{F}$ with  
		$$
		\mu_{F}(C_{n}^i ) \leq (1 - \gamma_{0}) m_{\overline{\varphi}}(C_{n}^i ) \quad\text{or}\quad m_{\overline{\varphi}}(C_{n}^i ) \leq (1 - \gamma_{0}) \mu_{F}(C_{n}^i ), 
		$$
		we have   $h_{\mu_{F}}(F) + \int \overline{\varphi} d \mu_{F} < \delta_{n,i}$.
	\end{proposition}

	In order to prove this result, we will first show that $m_{\overline{\varphi}}(C_{n} )$ decreases exponentially fast with $n$. It will allow us to choose $k$ in such way that  $m_{\overline{\varphi}}(C_{n} )$ is small enough. After that, we compute the Gurevich pressure for a modified potential and use it to estimate the free energy of measures that give the cylinders small mass, when compared to $m_{\overline{\varphi}}(C_{n} )$. Finally, we estimate the free energy of measures that gives the cylinders big mass, when compared to $m_{\overline{\varphi}}(C_{n} )$.  


\begin{lemma}
		\label{Intermediate}
		Suppose that we have an inducing scheme $(F,\mathcal{P})$ and a locally H\"older potential such that  $P_{G}(\overline{\varphi}) = 0$. If $m_{\overline{\varphi}}$ is the conformal measure 
		for the system $(F,\overline{\varphi})$, then for any $C_{n}  \in \mathcal{P}_{n}^{F}$ and any $n \in \mathbb{N}$,
		$$
		m_{\overline{\varphi}}(C_{n} ) \leq e^{-\lambda n},
		$$
		where $\lambda = -\log(K \sup_{C_{1}  \in \mathcal{P}_{1}^{F}} m_{\overline{\varphi}}(C_{1} ) )$ and $K= \exp\big{(} \sum_{j=1}^{\infty} V_{j}(\overline{\varphi})\big{)}$.
	\end{lemma} 
	
	\begin{proof}
		Since $m_{\overline{\varphi}}$ is a conformal measure, for $C_{n}  \in \mathcal{P}_{n}^{F}$ we have	
		$$
		1 = m_{\overline{\varphi}}(U_{\infty}) = m_{\overline{\varphi}}\left(\bigcup_{P\in\mathcal P}^{\infty}P \right) = m_{\overline{\varphi}}(F^{n}(C_{n} )) = \int_{C_{n} } e^{-S_{n}\overline{\varphi}} d m_{\overline{\varphi}},
		$$	
	since $d m_{\overline{\varphi}} = e^{S_{n}\overline{\varphi}} d m_{\overline{\varphi}} \circ F^{n}$. By the Intermediate Value Theorem we may choose $x \in C_{n} $ so that $e^{S_{n}\overline{\varphi}(x)} = m_{\overline{\varphi}} (C_{n} )$. Indeed, if either $e^{S_{n}\overline{\varphi}(x)} < m_{\overline{\varphi}} (C_{n} )$ or $e^{S_{n}\overline{\varphi}(x)} > m_{\overline{\varphi}} (C_{n} )$	for every $x \in C_{n} $, then we would have either $\int_{C_{n} } e^{-S_{n}\overline{\varphi}} d m_{\overline{\varphi}} > 1$ or $\int_{C_{n} } e^{-S_{n}\overline{\varphi}} d m_{\overline{\varphi}} < 1$. The continuity of $e^{-S_{n}\overline{\varphi}}$ guarantees the existence of $x$ as above.
		Therefore,
		$$
		m_{\overline{\varphi}}(C_{n} ) = e^{S_{n}\overline{\varphi}(x)} \leq e^{n\sup \overline{\varphi}}.
		$$
		By the Gibbs property,
		$$
		\displaystyle e^{\sup \overline{\varphi}} \leq K \sup_{C_{1} \in \mathcal{P}_{1}^{F}} m_{\overline{\varphi}}(C_{1}).
		$$
		Therefore
		$$
		\displaystyle \sup \overline{\varphi} \leq \log \Bigg{(} K \sup_{C_{1} \in \mathcal{P}_{1}^{F}} m_{\overline{\varphi}}(C_{1}) \Bigg{)},
		$$
	and so this can be chosen   the value for $-\lambda$.
	\end{proof}

	We call $K= \exp\big{(} \sum_{j=1}^{\infty} V_{j}(\overline{\varphi})\big{)}$
		  the \textit{distortion} of the potential $\overline\varphi$ with respect to the inducing scheme $(F, {\mathcal P})$. 
		  
		  \begin{remark}\label{re.large}
	Note that, given an   inducing scheme $(\tilde F, \tilde{\mathcal P})$ and  
	$k \in \mathbb{N}$, if the distortion of the potential $\overline{\varphi}$  with respect to   $(\tilde F, \tilde{\mathcal P})$ is bounded by some constant $K>0$, then the distortion of $\overline \varphi$ with respect to the inducing scheme $(F,\mathcal{P})$, where $F = \tilde{F}^{k}$ and $\mathcal P=\mathcal{P}^{\tilde F}_k$,  is still bounded by $K$. Lemma~\ref{Intermediate} shows that  
choosing $F = \tilde{F}^{k}$ and  $\mathcal P=\mathcal{P}^{\tilde F}_k$, for    $k$ large   we   have $\lambda$   large. 
	\end{remark}


	Note that as in \cite[Lemma 3]{Sa5},  we  have
	\begin{equation*}
P_{G}( {\varphi}) = 0\implies P_{G}(\overline{\varphi}) = 0.
\end{equation*}	
	In Lemma $\ref{Pressure}$ below, we will use the Variational Principle to bound the free energy of measures for the scheme which, for some $\gamma$, have  
	$\mu(C_{n}^i ) \leq Km_{\overline{\varphi}}(C_{n}^i )(1 - \gamma)/(1 - m_{\overline{\varphi}}(C_{n}^i ))^{n}$ in terms of the Gurevich pressure. However,
	instead of using $\overline{\varphi}$, which, in the computation of Gurevich pressure weights points $x \in C_{n}^{i}$ by $e^{\overline{\varphi}(x)}$, we use a potential
	which weights points in $C_{n}^{i}$ by $(1 - \gamma)e^{\overline{\varphi}(x)}$.	
	We define the potential $\overline{\varphi}_{n,i}^{\flat}$ by
	$$
	\overline{\varphi}_{n,i}^{\flat}(x) =
	\begin{cases}
		\overline{\varphi}(x) + \log (1 - \gamma), & \text{if }   x \in C_{n}^{i};\\
		\overline{\varphi}(x)  , & \text{if }  x \in C_{n}^{k}, \text{ with }   k \neq i .\\
	\end{cases}
	$$
In the sequel we will use   $A = \theta^{\pm C}$ to mean $\theta^{-C} \leq A \leq \theta^{C}$.

	\begin{lemma}
		\label{Gurevich}
		  $P_{G}(\overline{\varphi}_{n,i}^{\flat}) = \log (1 - \gamma m_{\overline{\varphi}}(C_{n}^i))$. 
	\end{lemma}
	
	\begin{proof}
		We assume  $n=1$, since the general case follows similarly. We will estimate $Z_{j}(\overline{\varphi}_{1,i}^{\flat}, C_{1}^{i})$, where $Z_{j}$ is defined by
		
		$$
		Z_{j}(\Phi,C) := \sum_{x :  \sigma^{j}(x) = x}\exp(S_{j}\Phi(x)) \chi_{C}(x) 
		$$
		The ideas we use here are similar to those in the proof of Claim 2 in the proof of \cite[Proposition 2]{BT}. It follows  from the definition that
		$$
		\displaystyle Z_{j}(\overline{\varphi}_{1,i}^{\flat}, C_{1}^{i}) = e^{\pm \sum_{k=0}^{j-1}V_{k}(\overline{\varphi})} \sum_{C_{j} \in \mathcal{P}_{j}^{F} \cap C_{1}^{i}}
		 \sum_{x \in C_{j}} e^{S_{j}\overline{\varphi}_{1,i}^{\flat}(x)}.
		$$		
		As in the proof of Lemma $\ref{Gibbs}$, the conformality of $m_{\overline{\varphi}}$ and the Intermediate Value Theorem imply that for any $k$ there is 
		$x_{C_{1}^{k}} \in C_{1}^{k}$ such that $m_{\overline{\varphi}}(C_{1}^{k}) = e^{\overline{\varphi}(x_{C_{1}^{k}})}$. Writing 
		$\overline{\varphi}_{k}  = \overline{\varphi}(x_{C_{1}^{k}})$ and $\overline{\varphi}_{k}^{\flat} = \overline{\varphi}_{1,i}^{\flat}(x_{C_{1}^{k}})$, we have $e^{\overline{\varphi}_{i}^{\flat}} = (1 - \gamma) e^{\overline{\varphi}_{i}}$.
		Therefore,	
		$$
		\displaystyle \sum_{k} e^{\overline{\varphi}_{k}^{\flat}} = \sum_{k \neq i} e^{\overline{\varphi}_{k}} + (1-\gamma)e^{\overline{\varphi}_{i}} = \sum_{k \neq i} m_{\overline{\varphi}}(C_{1}^{k}) +   (1-\gamma)e^{\overline{\varphi}_{i}} = 1-m_{\overline{\varphi}}(C_{1}^{i}) +   (1-\gamma)e^{\overline{\varphi}_{i}} = 1 - \gamma e^{\overline{\varphi}_{i}}.
		$$
		For any $C_{j} \in \mathcal{P}_{j}^{F}$ and  any $k \in \mathbb{N}$, there exists a unique $C_{j+1} \subset C_{j}$ such that $F^{j}(C_{j+1}) = C_{1}^{k}$.
		Moreover, there exists $x_{C_{j+1}} \in C_{j+1}$ such that $F^{j}(x_{C_{j+1}}) = x_{C_{1}^{k}}$. Then, for $C_{j} \subset C_{1}^{i}$,	
		$$
		\displaystyle \sum_{C_{j+1} \subset C_{j}} e^{S_{j+1} \overline{\varphi}_{1,i}^{\flat}(x_{C_{j+1}})} = e^{\pm V_{j+1}(\overline{\varphi})}
		e^{S_{j} \overline{\varphi}_{1,i}^{\flat}(x_{C_{j}})} \Bigg{(} \sum_{k} e^{\overline{\varphi}_{k}^{\flat}}\Bigg{)}
		= e^{\pm V_{j+1}(\overline{\varphi})}
		e^{S_{j} \overline{\varphi}_{1,i}^{\flat}(x_{C_{j}})} (1 - \gamma e^{\overline{\varphi}_{i}}).
		$$
		Therefore,
		$$
		Z_{j+1}(\overline{\varphi}_{1,i}^{\flat},C_{1}^{i}) = (1 - \gamma e^{\overline{\varphi}_{i}})
		e^{\pm \big{(}V_{j+1}(\overline{\varphi}) + \sum_{k=0}^{j-1}V_{k}(\overline{\varphi})\big{)}} Z_{j}(\overline{\varphi}_{1,i}^{\flat},C_{1}^{i}),
		$$
		hence
		$$
		Z_{j+1}(\overline{\varphi}_{1,i}^{\flat},C_{1}^{i}) = (1 - \gamma e^{\overline{\varphi}_{i}})^{j} e^{\pm \sum_{k=0}^{j} (k+1)V_{k}(\overline{\varphi})}.
		$$
		We remind that $\overline{\varphi}$ is weakly H\"older, so $\sum_{k=0}^{j} (k+1)V_{k}(\overline{\varphi}) < \infty$. Therefore we have 
		$P_{G}(\overline{\varphi}_{1,i}^{\flat}) = \log(1 - \gamma e^{\overline{\varphi}_{i}}) = \log(1 - \gamma m_{\overline{\varphi}}(C_{1}^{i}))$, thus proving the lemma.
	\end{proof}

	Let $\mathcal{M}_{F}(\overline{\varphi})$ denote the set of $F$-invariant probability measures such that $-\int \overline{\varphi} d \mu < \infty$.

	\begin{lemma}
		\label{Pressure}
		We have   $\mathcal{M}_{F}(\overline{\varphi}) = \mathcal{M}_{F}(\overline{\varphi}_{n,i}^{\flat})$ and, for any cylinder $C_{n}^{i} \in \mathcal{P}_{n}^{F}$,
		\begin{eqnarray*} 
			 \sup & &\hspace{-.9cm} \left\{  h_{F}(\mu) + \int \overline{\varphi} d \mu : \mu \in \mathcal{M}(\overline{\varphi}),\; \mu(C_{n}^{i}) < \frac{K(1 - \gamma)}{1 - m_{\overline{\varphi}}(C_{n}^{i})^{n}}
			m_{\overline{\varphi}}(C_{n}^{i}) \right\} \\
			& \leq &
			\sup\left\{ h_{F}(\mu) + \int \overline{\varphi}_{n}^{\flat} d \mu : \mu \in \mathcal{M}(\overline{\varphi}_{n,i}^{\flat}),\; \mu(C_{n}^{i}) < \frac{K(1 - \gamma)}{1 - m_{\overline{\varphi}}(C_{n}^{i})^{n}}
			m_{\overline{\varphi}}(C_{n}^{i}) \right\}  \\
			&& - \left[ \frac{K(1 - \gamma)\log(1-\gamma)}{1 - m_{\overline{\varphi}}(C_{n}^{i})^{n}}\right]
			m_{\overline{\varphi}}(C_{n}^{i})  \\
			& \leq &
			P_{G}(\overline{\varphi}_{n,i}^{\flat}) - \left[ \frac{K(1 - \gamma)\log(1-\gamma)}{1 - m_{\overline{\varphi}}(C_{n}^{i})^{n}}
			m_{\overline{\varphi}}(C_{n}^{i})\right]. 
		\end{eqnarray*}
	\end{lemma}
	
	\begin{proof}
	The fact that $\mathcal{M}_{F}(\overline{\varphi}) = \mathcal{M}_{F}(\overline{\varphi}_{n,i}^{\flat})$ is clear from the definition.
	Suppose now that $\mu \in \mathcal{M}_{F}(\overline{\varphi})$ and $\mu(C_{n}^{i}) \leq m_{\overline{\varphi}}(C_{n}^{i})K(1 - \gamma)/(1 - m_{\overline{\varphi}}(C_{n}^{i}))^{n}$.
	Then
	\begin{align*}
	\left(h_{F}(\mu) + \int \overline{\varphi} d\mu \right) - \left(h_{F}(\mu) + \int \overline{\varphi}_{n}^{\flat} d\mu \right) 
	&= 
	\int (\overline{\varphi} - \overline{\varphi}_{n}^{\flat}) d \mu\\
	&= \mu(C_{n}^{i})\left(-\log(1 - \gamma)\right) \\
	&\leq - \left[ \frac{K(1 - \gamma)\log(1-\gamma)}{1 - m_{\overline{\varphi}}(C_{n}^{i})^{n}}
	m_{\overline{\varphi}}(C_{n}^{i})\right].
\end{align*}
This proves  the first inequality in the statement. The second  inequality follows from the definition of pressure.
	\end{proof}

	We can actually prove that the last  inequality is   an equality, but we do not need it
	 here. Lemmas $\ref{Gurevich}$ and $\ref{Pressure}$ give that for $C_{n}^{i} \in \mathcal{P}_{n}^{F}$ with $\mu(C_{n}^{i}) < K(1 - \gamma)m_{\overline{\varphi}}(C_{n}^{i})/
	(1 - m_{\overline{\varphi}}(C_{n}^{i})^{n})$, we   have
	\begin{eqnarray}
		h_{F}(\mu) + \int \overline{\varphi}_{n,i}^{\flat} d \mu \nonumber
		&\leq &
		P_{G}(\overline{\varphi}_{n,i}^{\flat}) - \bigg{[} \frac{K(1 - \gamma)\log(1-\gamma)}{1 - m_{\overline{\varphi}}(C_{n}^{i})^{n}}
		m_{\overline{\varphi}}(C_{n}^{i})\bigg{]}
		\\
		\label{eqn2}
		&= & \log (1 - \gamma m_{\overline{\varphi}}(C_{n}^i)) - \bigg{[} \frac{K(1 - \gamma)\log(1-\gamma)}{1 - m_{\overline{\varphi}}(C_{n}^{i})^{n}}
		m_{\overline{\varphi}}(C_{n}^{i})\bigg{]}.
	\end{eqnarray}
	Now, if $m_{\overline{\varphi}}(C_{n}^{i})$ is sufficiently small, then $\log (1 - \gamma m_{\overline{\varphi}}(C_{n})) \approx - \gamma m_{\overline{\varphi}}(C_{n})$. Moreover, by Lemma~\ref{Intermediate}, we have $m_{\overline{\varphi}}(C_{n}^{i}) < e^{-\lambda n}$. Therefore, choosing $k$  large (which gives~$\lambda$ large, by Remark~\ref{re.large}), we have for some choice of $  \tilde{\gamma}^{\flat} \in (0,1)$ sufficiently close to 1,  
%
%
	\begin{equation}\label{deltaflat}
	\delta_{n,i}^{\flat} := \log (1 - \tilde{\gamma}^{\flat} m_{\overline{\varphi}}(C_{n}^i)) - \bigg{[} \frac{K(1 - \tilde{\gamma}^{\flat})\log(1-\tilde{\gamma}^{\flat})}{1 - m_{\overline{\varphi}}(C_{n}^{i})^{n}}
	m_{\overline{\varphi}}(C_{n}^{i})\bigg{]}<0.
\end{equation}

	For the upper bound on the free energy of measures giving $C_{n}^{i}$ relatively large mass, we follow a similar proof, but with the potential
	$\overline{\varphi}_{n,i}^{\sharp}$ defined as 
	$$
	\overline{\varphi}_{n,i}^{\sharp}(x) =  
	\begin{cases}
		\overline{\varphi}(x) - \log (1 - \gamma), & \text{if} \,\, x \in C_{n}^{i};  \\
		\overline{\varphi}(x),   & \text{if} \,\, x \in C_{n}^{k}, \text{ with }  k \neq i .\\
	\end{cases}
	$$
	Similarly to Lemmas $\ref{Gurevich}$ and $\ref{Pressure}$ above, we can prove the next two lemmas.

	\begin{lemma}
		    $P_{G}(\overline{\varphi}_{n,i}^{\sharp}) = \log \Big{(}1 +  m_{\overline{\varphi}}(C_{n})\Big{(}\frac{\gamma}{1-\gamma}\Big{)}\Big{)}
		    $. 
	\end{lemma}

	 \begin{lemma}
		We have  $\mathcal{M}_{F}(\overline{\varphi}) = \mathcal{M}_{F}(\overline{\varphi}_{n}^{\sharp})$ and, for any cylinder $C_{n}^{i} \in \mathcal{P}_{n}^{F}$,  
		\begin{align*}
			 \sup&\left\{ h_{F}(\mu) + \int \overline{\varphi} d \mu : \mu \in \mathcal{M}(\overline{\varphi}), \;
			\mu(C_{n}^{i}) >
			 \frac{m_{\overline{\varphi}}(C_{n}^{i}) }{K(1-\gamma)\left(1 +  m_{\overline{\varphi}}(C_{n})\left(\frac{\gamma}{1-\gamma}\right)\right)^{n}}
			\right\}  \\
			\leq &
			\sup \left\{ h_{F}(\mu) + \int \overline{\varphi}_{n,i}^{\sharp} d \mu : \mu \in \mathcal{M}(\overline{\varphi}_{n,i}^{\sharp}), \;\mu(C_{n}^{i}) 
			> \frac{m_{\overline{\varphi}}(C_{n}^{i})}{K(1-\gamma)\left(1 +  m_{\overline{\varphi}}(C_{n})\left(\frac{\gamma}{1-\gamma}\right)\right)^{n}}
			 \right\} + \\
			& +  \left[ \frac{\log(1-\gamma)m_{\overline{\varphi}}(C_{n}^{i})}{K(1-\gamma)\left(1 +  m_{\overline{\varphi}}(C_{n})\left(\frac{\gamma}{1-\gamma}\right)\right)^{n}}\right]\\
			\leq &
			P_{G}(\overline{\varphi}_{n,i}^{\sharp}) + \left[ \frac{\log(1-\gamma)m_{\overline{\varphi}}(C_{n}^{i})}{K(1-\gamma)\left(1 +  m_{\overline{\varphi}}(C_{n})\left(\frac{\gamma}{1-\gamma}\right)\right)^{n}}\right].\\
		\end{align*}
	\end{lemma}
It follows from the last two lemmas that  if
	$$
	\mu(C_{n}^{i}) > \frac{m_{\overline{\varphi}}(C_{n})}{K(1-\gamma)\Big{(}1 +  m_{\overline{\varphi}}(C_{n})\Big{(}\frac{\gamma}{1-\gamma}\Big{)}\Big{)}^{n}},
	$$
	then  
	\begin{equation}
		\label{eqn3}
		h_{F}(\mu) + \int \overline{\varphi} d\mu \leq m_{\overline{\varphi}}(C_{n}^{i})\bigg{(} \frac{\gamma}{1 - \gamma} \bigg{)} + 
		\frac{\log(1-\gamma)m_{\overline{\varphi}}(C_{n}^{i})}{K(1-\gamma)\Big{(}1 +  m_{\overline{\varphi}}(C_{n})\Big{(}\frac{\gamma}{1-\gamma}\Big{)}\Big{)}^{n}}.
	\end{equation}
Similarly to what we did for obtaining~\eqref{deltaflat}, taking  $k$ large  (which gives $\lambda$ large), we are also able to choose $ \tilde{\gamma}^{\sharp} \in (0,1)$ close to 1 so that 
%
%
	$$
	\delta_{n,i}^{\sharp} := m_{\overline{\varphi}}(C_{n}^{i})\bigg{(} \frac{\tilde{\gamma}^{\sharp}}{1 - \tilde{\gamma}^{\sharp}} \bigg{)} + 
	\frac{\log(1-\tilde{\gamma}^{\sharp})m_{\overline{\varphi}}(C_{n}^{i})}{K(1-\tilde{\gamma}^{\sharp})\Big{(}1 +  m_{\overline{\varphi}}(C_{n}^i)\Big{(}\frac{\tilde{\gamma}^{\sharp}}{1-\tilde{\gamma}^{\sharp}}\Big{)}\Big{)}^{n}}<0.
	$$
	This can be seen as follows: let $\tilde{\gamma}^{\sharp} = p/(p + 1)$ for some $p$ to be chosen later.
	Then $\delta_{n}^{i,\sharp}$ becomes
	\begin{eqnarray}
		\label{eqn4}
		\delta_{n,i}^{\sharp} = m_{\overline{\varphi}}(C_{n}^{i})(p + 1)\bigg{[}\frac{p}{p+1} - \frac{\log(p+1)}{K(1 + pm_{\overline{\varphi}}(C_{n}^i))^{n}}\bigg{]}. 
	\end{eqnarray}
	If $\lambda$ is sufficiently large and assuming  $
		m_{\overline{\varphi}}(C_{n}^i ) \leq e^{-\lambda n},
		$   there exists some large $\lambda'<\lambda$ such that
	 $$(1 + p m_{\overline{\varphi}}(C_{n}^i ))^{n}\le (1 + pe^{-\lambda n})^{n} \leq 1 + pe^{-\lambda' n},\quad\forall n \in \mathbb{N}.
	$$
	Hence,   we can choose $p$  so that the quantity in the square brackets in \eqref{eqn4} is negative for all $n$, which then gives $\delta_{n,i}^{\sharp}<0$.

Now, set
	\begin{eqnarray}
		\label{eqn5}
		\gamma^{\sharp} = 1 - (1 - \tilde{\gamma}^{\sharp}) \Bigg{(} 1 + e^{-\lambda n}\Bigg{(}\frac{\tilde{\gamma}^{\sharp}}{1 - \tilde{\gamma}^{\sharp}}\Bigg{)}\Bigg{)}^{n}.
	\end{eqnarray}
	For sufficiently large   $\lambda$, we have   $\gamma^{\sharp}\in (0,1)$.
	Let  $$\gamma_{0}':= \max\{\gamma^{\flat}, \gamma^{\sharp}\}$$ and, for each $C_{n}^{i} \in \mathcal{P}_{n}^{F}$, 
	$$\delta_{n,i}: = \max\{\delta_{n,i}^{\flat},\delta_{n,i}^{\sharp}\}.$$ 
	The proof of Proposition~\ref{Gibbs} is complete,   setting $$\gamma_{0}:=1 - K(1 - \gamma_{0}'),$$
	which we may assume belongs in $(0,1)$. 
	
	
	\subsection{Integrability of the inducing time}
	
	To finish the proof of Theorem~\ref{A} we  use the next result, which is essentially a consequence of Theorem~\ref{Pinheiro}, Theorem~\ref{Sarigseveral} and  Proposition \ref{Gibbs}.
	
	\begin{proposition}
		\label{Finitely}
		Given a hyperbolic potential $\varphi$ with $P_{f}(\varphi)=0$, there exist finitely many ergodic equilibrium states for the system $(f,\varphi)$ and they are expanding.
		
	\end{proposition}
	
	In order to prove this result, we take a sequence of $f$-invariant measures $\{\mu_{n}\}_{n}$ such that $h_{\mu_{n}}(f) + \int \varphi d \mu_{n} \to 0$, each $\mu_n$   liftable to a measure $\overline{\mu}_n$
that, with no loss of generality, we may assume all	liftings with respect to a same inducing scheme  $(\tilde F, \mathcal P)$. We will show that the set of lifted measures 
	has the Gibbs measure as its unique accumulation point with respect to which the inducing time is integrable.
	Finally, this Gibbs measure is projected to an equilibrium state for $(f,\varphi)$.
	
	We take the inducing scheme $(\tilde{F},\mathcal{P})$ with inducing time $\tilde{\tau}$ and, for   $k \in \mathbb{N}$ large enough, we consider an iterated inducing scheme $(F,\mathcal{P}_{k}^{\tilde{F}})$, with $F = \tilde{F}^{k}$, so that the conclusions of Proposition \ref{Gibbs} are satisfied. Recall that
	\[
	\mathcal{P}_{k}^{\tilde{F}} = \bigvee_{j=0}^{k-1} \tilde{F}^{-j}(\mathcal{P}).
	\]
	We say that a set $\mathcal{K}$ of measures  on the domain $X$ of $F$ is {\textit{tight}} if, for every $\epsilon > 0$, there exists a compact set $K \subset X$ 
	such that $\eta(K^{c}) < \epsilon$, for every measure $\eta \in \mathcal{K}$.
	Prohorov Theorem asserts that any sequence of measures in a tight set has some accumulation point in the weak* topology; see e.g. \cite[Theorem 2.1.8]{OV}.
	
	\begin{lemma}
		\label{Tight}
		The lifted sequence $\{\overline{\mu}_{n}\}_{n}$  is tight.
	\end{lemma}
	
	\begin{proof}
		
		Recalling that $\tilde{\tau}$ is the inducing time of $\tilde{F}$, by   Theorem~\ref{Pinheiro} there exists $\theta > 0$ such that $\int \tilde{\tau} d \overline{\mu}_{n} < \theta, $ for all $ n \in \mathbb{N}$. We claim that this implies that the set $\{\overline{\mu}_{n}\}$ is tight. It is enough to show that, given $j \in \mathbb{N}$, we can find a compact set $K_{j}$ such that  $\overline{\mu}_{n}(K_{j}^{c}) <  {\theta}/{j},$ for all  $n \in \mathbb{N}$. In fact,
		$$
		j \overline{\mu}_{n} (\{\tilde{\tau} > j\}) < \int_{\{\tilde{\tau} > j\}} \tilde{\tau} d \overline{\mu}_{n} < \theta \implies \overline{\mu}_{n}(\{\tilde{\tau} > j\}) < \frac{\theta}{j},\quad \forall n \in \mathbb{N}.
		$$
		It remains to show that $K_{j} = \{\tilde{\tau} \leq j\}$ is compact. Indeed, $K_{j}$ is the union of finitely many cylinders, which are compact sets. So, the set $\{\overline{\mu}_{n}\}$ is tight, as  claimed.  
	\end{proof}
	
	Let $\tau$ be the return time associated with the inducing scheme $(F,\mathcal{P}_{k}^{\tilde{F}})$ and  $\overline\varphi$ be the lifting of  $\varphi$ to this inducing scheme. It follows from Theorem~\ref{Sarigseveral} that $\overline\varphi$ has a unique equilibrium state $\mu_{\overline\varphi}$, which is a Gibbs measure.
	
	    \begin{lemma}\label{integrable}
	    	 $\tau$  is integrable with respect to $\mu_{\overline{\varphi}}$. 
	    \end{lemma} 
        \begin{proof}	
	 By   Proposition \ref{Gibbs}, there exists $K'>0$ such that, given a cylinder $C_{n} \in \mathcal{P}_{n}^{F}$, there exists $k_{n} \in \mathbb{N}$ such that, for all $k \geq k_{n}$, we have
		$$
		\frac{1}{K'} \leq \frac{\overline{\mu}_{k}(C_{n})}{e^{S_{k}\overline{\varphi}(x)}} \leq K', \quad \forall x \in C_{n}.
		$$
		By Lemma \ref{Tight}, the set $\{\overline{\mu}_{n}\}$ is tight and so, by Prohorov Theorem, it has a subsequence converging in the weak* topology, which we keep writing $\{\overline{\mu}_{n}\}_{n}$ with no loss of generality. 
		It is not difficult to see that the limit is a Gibbs measure and, by uniqueness, it must be the measure $\mu_{\overline{\varphi}}$.
		
		Now, we   see that the inducing time $\tilde{\tau}$ is integrable with respect to $\mu_{\overline{\varphi}}$. 
		First of all, recall that $F = \tilde{F}^{k}$ and we may have two lifts of $\mu_n$, depending on whether we consider the inducing scheme $(\tilde F,\mathcal P)$ or~$(F,\mathcal P_k)$.  Let $\overline{\mu}_{F,n}$ denote the lift  with respect to  $F$ and $\overline{\mu}_{\tilde{F},n}$ denote  the lift  with respect to   $\tilde{F}$. We have
		$\int \tau d \overline{\mu}_{F,n} = \int \tilde{\tau}^{k} d \overline{\mu}_{\tilde{F},n} \leq \theta k$. Setting $\tau_{N} = \min\{\tau,N\}$, by the Monotone Convergence Theorem, we obtain		
		$$
		\displaystyle
		\int \tau d \mu_{\overline{\varphi}} =
		\lim_{N \to \infty} \int \tau_{N} d \mu_{\overline{\varphi}} \leq
		\lim_{N \to \infty} \limsup_{n \to \infty}\int \tau_{N} d \overline{\mu}_{F,n} \leq \theta k.
		$$
		This gives that $\tau$ is integrable with respect to $\mu_{\overline{\varphi}}$.
		\end{proof}
		
		\begin{proof}[Proof of Proposition~\ref{Finitely}]
		Since $\tau$ is integrable with respect to $\mu_{\overline{\varphi}}$, by Lemma \ref{integrable}, we can project the Gibbs measure $\mu_{\overline{\varphi}}$ and obtain an invariant measure $\mu$ for $f$. By the Abramov's formulas we   see that $\mu$ is an equilibrium state for the system $(f,\varphi)$.	
		As there exists an equilibrium state, we also can find an ergodic one. Also, if $\nu$ is an ergodic equilibrium state, we can lift it to an equilibrium state for the shift, which is the Gibbs measure $\mu_{\overline{\varphi}}$. So, the projection of it is $\nu$. This shows that there exists at most one ergodic equilibrium state for each inducing scheme. Then, there are  finitely many ergodic equilibrium states for the system $(f,\varphi)$. 
	\end{proof}
	
\section{Hyperbolic and Expanding Potentials}

In this section, we prove Theorems \ref{B} and \ref{C}. We will call the $(\sigma,\epsilon)$-hyperbolic potentials simply as hyperbolic potentials and we call the $(\sigma,\epsilon)$-expanding potentials simply as expanding potential, since we assume that $\sigma$ and $\epsilon$ are previously fixed. We show that class of expanding potentials is equivalent to the class of hyperbolic potentials. This class was inspired by the work \cite{PV}, where the authors consider a type of potential that they call expanding and also announced the uniqueness of measure of maximal entropy for Viana maps. Moreover, by using a result for Viana maps in \cite{ALP}, where the authors announced the existence of countably many ergodic measures of maximal entropy for Viana maps, we show that the null potential is expanding if the $SRB$ measure is not a measure of maximal entropy, and so it is hyperbolic, giving finitely many ergodic measures of maximal entropy. Also, it implies that the potentials whose integral with respect to every invariant measure vanishes are expanding (and hyperbolic). If the SRB measure is a measure of maximal entropy, we can also use the strategy of the proof of Theorem \ref{A} to obtain finitely many ergodic measures of maximal entropy, having the SRB measure in the finite set of ergodic measures of maximal entropy. 

We recall the definition of an expanding potential (subsection \ref{expanding}). Let $\phi:M \to \mathbb{R}$ be a continuous potential. We say that $\phi$ is an expanding potential if the following inequality holds
\[
\displaystyle \sup_{\mu \in \mathcal{E}^{c}} \bigg{\{} h_{\mu}(f) + \int \phi d \mu \bigg{\}} < \sup_{\mu \in \mathcal{E}} \bigg{\{} h_{\mu}(f) + \int \phi d \mu \bigg{\}},	
\]
where $\mathcal{E}$ denotes the set of all expanding measures for $f$. We observe that Proposition \ref{pr.hypmeas} shows that every hyperbolic potential is an expanding potential. Since for an expanding potential $\phi$ we can find a sequence $\mu_{n}$ of expanding measures such that   $h_{\mu_{n}}(f) + \int \phi d \mu_{n} \to P_{f}(\phi)$, if $\phi$ is also H\"older, we also can find finitely many ergodic equilibrium states for~$\phi$, which are expanding measures. In fact, all the proof is similar from the point where Proposition \ref{pr.hypmeas} is proved. We intend to prove the following result.

\begin{theorem}\label{HYPERZOOM}
	Let $f:M \to M$ be a continuous map with an expanding measure and $\phi : M \to \mathbb{R}$ a continuous potential. Then $\phi$ is   hyperbolic   if, and only if, it is   expanding. In particular, if $\phi$ is hyperbolic and H\"older continuous with   $P_{f}(\phi)<\infty$, then it has finitely many equilibrium states which are expanding measures.
\end{theorem}

Denote   $\mathcal{HP}$ the set of hyperbolic potentials and $\mathcal{EP}$ the set of expanding potentials. We divide the proof of Theorem \ref{HYPERZOOM} into some lemmas. The first lemma gives that both the sets of hyperbolic and expanding potentials are open in the topology of the supremum norm. We use the following property of the relative pressure given in \cite{Pe2}: for any continuous potentials $\phi, \varphi: M \to \mathbb{R}$ and   $X \subset M$, we have
	\begin{equation}\label{eq.pesin}
|P(\phi,X) - P(\varphi,X)| \leq \| \phi - \varphi \|_{\infty}.
\end{equation}

\begin{lemma} \label{OPEN}
	  $\mathcal{HP}$ and $\mathcal{EP}$ are open in the topology of the supremum norm $\| \cdot \|_{\infty}$. 
\end{lemma}
\begin{proof}
	First, we prove that  $\mathcal{HP}$ is open. Assuming that $\phi$ is hyperbolic and denoting $\Lambda  = H(\sigma,\epsilon)$, we have   $P(\phi,\Lambda^{c}) < P(\phi, \Lambda)$. Take $\epsilon>0$ such that $0 < 2\epsilon < P(\phi, \Lambda) - P(\phi,\Lambda^{c})$ and $\varphi$ such that $\| \phi - \varphi \|_{\infty} < \epsilon$. Using~\eqref{eq.pesin}, we get 
\begin{align*}
P(\varphi, \Lambda) - P(\varphi,\Lambda^{c}) 
&= 
(P(\varphi, \Lambda) - P(\phi,\Lambda)) + (P(\phi,\Lambda) - P(\phi,\Lambda^{c})) + (P(\phi,\Lambda^{c}) - P(\varphi,\Lambda^{c})) \\
&> 
	-\| \phi - \varphi \|_{\infty} + P(\phi,\Lambda) - P(\phi,\Lambda^{c}) - \| \phi - \varphi \|_{\infty}\\
	& > P(\phi,\Lambda) - P(\phi,\Lambda^{c}) - 2\epsilon > 0.
\end{align*}
	This implies that $\varphi$ is hyperbolic and, as a consequence, we have the ball $B_{\epsilon}(\phi) \subset \mathcal{HP}$. This shows that the set $\mathcal{HP}$ is open.	
	In order to show that the set $\mathcal{EP}$ is open, we assume that the potential $\phi$ is expanding, that is,
	\[
	\displaystyle  \sup_{\eta \in \mathcal{E}^{c}} \bigg{\{} h_{\eta}(f) + \int \phi d\eta \bigg{\}} < \sup_{\eta \in \mathcal{E}} \bigg{\{} h_{\eta}(f) + \int \phi d\eta \bigg{\}}.
	\]
	So, taking
	\[
	\displaystyle 0 < 2\epsilon < \sup_{\eta \in \mathcal{E}} \bigg{\{} h_{\eta}(f) + \int \phi d\eta \bigg{\}} - \sup_{\eta \in \mathcal{E}^{c}} \bigg{\{} h_{\eta}(f) + \int \phi d\eta \bigg{\}},
	\]
	and $\varphi$ such that $\| \phi - \varphi \|_{\infty} < \epsilon$ we have
	\begin{eqnarray*}
  \sup_{\eta \in \mathcal{E}}  \left\{ h_{\eta}(f) + \int \varphi d\eta \right\} \!&-& \!\sup_{\eta \in \mathcal{E}^{c}} \bigg{\{} h_{\eta}(f) + \int \varphi d\eta \bigg{\}}   
	\\
		&=&\sup_{\eta \in \mathcal{E}} \bigg{\{} h_{\eta}(f) + \int \varphi d\eta \bigg{\}} - \sup_{\eta \in \mathcal{E}} \bigg{\{} h_{\eta}(f) + \int \phi d\eta \bigg{\}} 
		\\
		&& +
	   \sup_{\eta \in \mathcal{E}} \bigg{\{} h_{\eta}(f) + \int \phi d\eta \bigg{\}} - \sup_{\eta \in \mathcal{E}^{c}} \bigg{\{} h_{\eta}(f) + \int \phi d\eta \bigg{\}} \\
	&&	+\sup_{\eta \in \mathcal{E}^{c}} \bigg{\{} h_{\eta}(f) + \int \phi d\eta \bigg{\}} - \sup_{\eta \in \mathcal{E}^{c}} \bigg{\{} h_{\eta}(f) + \int \varphi d\eta \bigg{\}} \\
	&>&   
	  \sup_{\eta \in \mathcal{E}} \bigg{\{} h_{\eta}(f) + \int \phi d\eta \bigg{\}} - \sup_{\eta \in \mathcal{E}^{c}} \bigg{\{} h_{\eta}(f) + \int \phi d\eta \bigg{\}} - 2\epsilon > 0.
\end{eqnarray*}
We used that $\| \phi - \varphi \|_{\infty} < \epsilon$ implies $\int \varphi d\eta - \int \phi d\eta > -\epsilon$ and also $\int \phi d\eta - \int \varphi d\eta > -\epsilon$. Thus, we have that $\varphi$ is expanding and, as a consequence, we have the ball $B_{\epsilon}(\phi) \subset \mathcal{EP}$. It shows that the set $\mathcal{EP}$ is open.
\end{proof}

Recall that Proposition \ref{pr.hypmeas} guarantees that $\mathcal{HP} \subset \mathcal{EP}$. In order to prove Theorem~\ref{HYPERZOOM} it remains to show that $\mathcal{EP} \subset \mathcal{HP}$. In the next lemma, we show this by showing that $\mathcal{EP} \subset \overline{\mathcal{HP}}$ and using that $\mathcal{HP} \subset \mathcal{EP}$ are open sets.

%
%

\begin{lemma}\label{ZOOMHYPER}
	We have that $\mathcal{EP} \subset \overline{\mathcal{HP}}$. 
\end{lemma}
\begin{proof}
	Let $\phi \in \mathcal{EP}$ be a H\"older potential, which exists because $\mathcal{EP}$ is open and the set of H\"older functions are dense in the set of continuous functions. By Theorem \ref{A}, there exist finitely many ergodic equilibrium states, which are also expanding measures. Let $\mu_{0}$ be an equilibrium state. If $\phi \in \mathcal{EP} \backslash \mathcal{HP}$, denoting $\Lambda:=H(\sigma,\epsilon)$ we have the following inequalities
	\[
	P(\phi,\Lambda) \leq P(\phi,\Lambda^{c}) = P(\phi) = h_{\mu_{0}} + \int \phi d\mu_{0} \leq \sup_{\eta \in \mathcal{E}} \bigg{\{} h_{\eta}(f) + \int \phi d\eta \bigg{\}} \leq P(\phi,\Lambda),  
	\]
	because $\mu_{0}$ is expanding and the last inequality is a consequence of  \cite[Theorem A2.1]{Pe2}. Hence, we obtained that $P(\phi,\Lambda) = P(\phi,\Lambda^{c}) = P(\phi)$. Since by \eqref{eq.pesin} we have that the relative pressure is continuous, it follows that $\overline{\mathcal{HP}}$ is the set of continuous potentials $\varphi$ such that $P(\varphi,\Lambda^{c}) \leq P(\varphi,\Lambda)$. We conclude that $\phi \in \overline{\mathcal{HP}}$. Then, every H\"older expanding potential belongs  in $\overline{\mathcal{HP}}$. By denseness of the H\"older functions, we have   $\mathcal{EP} \subset \overline{\mathcal{HP}}$, as desired.
\end{proof}

By Lemma \ref{ZOOMHYPER}, we have that $\mathcal{EP} \subset \overline{\mathcal{HP}}$. Since $\mathcal{EP}$ is an open set, we conclude that $\mathcal{EP} \subset \mathcal{HP}$ and, finally, $\mathcal{EP} = \mathcal{HP}$, concluding the proof of Theorem \ref{HYPERZOOM}.

\begin{proposition}\label{ZER}
	Let $f:S^{1} \times I \to S^{1} \times I$ be a Viana map. If the unique $SRB$ measure is not a measure of maximal entropy, then the null potential is expanding and, by Theorem~\ref{HYPERZOOM} it is hyperbolic. 
\end{proposition}

In particular, there exist finitely many ergodic measures of maximal entropy. If the SRB measure is a measure of maximal entropy, we can also use the strategy of the proof of Theorem \ref{A} to obtain this, having the SRB measure in the finite set of ergodic measures of maximal entropy.

\begin{proof}
	In \cite[Proposition 12.2]{ALP} the authors establish that for Viana maps we have the following result, among others: if $\mu$ is an $f$-invariant measure such that $h_{\mu}(f) \geq h_{SRB}(f)$ where $SRB$ denotes the unique SRB measure for Viana maps, then the measure $\mu$ is hyperbolic. It implies that
	\[
	\sup_{\nu \in \mathcal{E}^{c}}\{h_{\nu}(f)\} \leq \sup_{\mu \in \mathcal{E}}\{h_{\mu}(f)\} = h(f).
	\]
	If the inequality is strict, it means that the null potential is expanding and Theorem \ref{HYPERZOOM} establishes the result in this case. Otherwise, we still can take a sequence of expanding measures $\mu_{n}$ such that $h_{\mu_{n}}(f) \to h(f)$ and the proof to find equilibrium states proceeds analogously. We then obtain finiteness of ergodic measures of maximal entropy in any case.
\end{proof}

\begin{remark}
	We   improve a result in \cite{ALP}, where the authors prove that for Viana maps there exist at most countably many ergodic measures of maximal entropy.
\end{remark}

\begin{corollary}\label{BIR}
	Let $\phi: M \to \mathbb{N}$ with its Birkhoff sums uniformly bounded, that is, there exists $r > 0$ such that
	\[
	|S_{n}\phi(x)| < r,\quad \forall n \in \mathbb{N} ,\;\forall x \in M.
	\]
	If the null potential is expanding, then $\phi$ is also expanding (and hyperbolic).
\end{corollary}
\begin{proof}
	By Birkhoff's Ergodic Theorem we have that $\int \phi d\eta = 0$ for every $f$-invariant probability $\eta$. By Proposition \ref{ZER} the null potential is expanding if the SRB measure is not a measure of maximal entropy. It implies that 
	\[
	\displaystyle  \sup_{\eta \in \mathcal{E}^{c}} \bigg{\{} h_{\eta}(f) + \int \phi d\eta \bigg{\}} = \sup_{\eta \in \mathcal{E}^{c}} \bigg{\{} h_{\eta}(f) \bigg{\}} < \sup_{\eta \in \mathcal{E}} \bigg{\{} h_{\eta}(f) \bigg{\}} = \sup_{\eta \in \mathcal{E}} \bigg{\{} h_{\eta}(f) + \int \phi d\eta \bigg{\}}.
	\]  
	So, $\phi$ is an expanding (and hyperbolic) potential.
\end{proof}

\begin{remark}
	The previous lemma proved that the potential $\phi$ is expanding (and hyperbolic) if we have $\int \phi d \eta = 0$ for every invariant measure $\eta$. In particular, if we have the Birkhoff sums uniformly bounded. Moreover, a similar result can be obtained if $f:M \to M$ has an expanding measure and the null potential is expanding (and hyperbolic).
\end{remark}

\bigskip

	\paragraph{\bf Acknowledgement}{The authors would like to thank Mike Todd for pointing out a mistake in the proof of Proposition~\ref{pr.todd}  in a preliminary version of the text. Also, we would like to thank Godofredo Iommi for clarifying a key step in a proof of   \cite{IT2}}. We also thank anonymous referees for their valuable comments to improve our manuscript.

\end{document}